\documentclass[11pt]{article}

\usepackage{latexsym}
\usepackage{amsfonts}
\usepackage{amscd}
\usepackage{amsmath, amsthm, amssymb}
\usepackage{bbm}
\usepackage[all]{xy}
\usepackage{cite}

\author{Thomas John Baird}
\title{Antiperfection of Yang-Mills Morse theory over a nonorientable surface}

\newtheorem{thm}{Theorem}[section]
\newtheorem{cor}[thm]{Corollary}
\newtheorem{lem}[thm]{Lemma}
\newtheorem{prop}[thm]{Proposition}

\theoremstyle{definition}

\newtheorem{rmk}{Remark}

\newcommand{\ignore}[1]{}

\newcommand{\lie}[1]{\mathfrak{#1}}

\newcommand{\rk}{\mathrm{rk}}
\newcommand{\im}{\mathrm{im}}
\newcommand{\id}{\mathbbmss{1}}
\newcommand{\Z}{\mathbb{Z}}
\newcommand{\R}{\mathbb{R}}
\newcommand{\C}{\mathbb{C}}

\newcommand{\Q}{\mathbb{Q}}

\newcommand{\mZ}{\mathcal{Z}}

\newcommand{\ti}[1]{\tilde{#1}}

\newcommand{\gau}{\mathcal{G}}
\newcommand{\A}{\mathcal{A}}
\newcommand{\dbar}{\bar{\partial}}
\newcommand{\cov}{\tilde{\Sigma}}
\newcommand{\pcov}{\ti{E}}

\newcommand{\B}{\mathcal{B}}
\newcommand{\CC}{\mathcal{C}}
\newcommand{\h}{h}

\newcommand{\slope}{\mathrm{sl}}
\newcommand{\gr}{\mathrm{gr}}

\begin{document}



\maketitle

\begin{abstract}
We use Morse theory of the Yang-Mills functional to compute the Betti numbers of the moduli stack of flat U(3)-bundles over a compact nonorientable surface. Our result establishes the antiperfection conjecture of Ho-Liu, and provides evidence for the equivariant formality conjecture of the author.
\end{abstract}

\section{Introduction}

Let $\Sigma$ denote a compact, connected 2-manifold without boundary and $E \rightarrow \Sigma$ a smooth, Hermitian $\C^n$-bundle. Denote by $\A = \A(E)$ the space of unitary connections, $\A^{\flat}$ the subspace of flat connections and by $\gau = \gau(E)$ the group of unitary gauge tranformations. The gauge group $\gau$ acts on $\A$ and restricts to an action on $\A^{\flat}$. In this paper we compute the Poincar\'e series of the equivariant cohomology ring $H^*_{\gau}(\A^{\flat})$ when $n=3$ and $\Sigma$ is nonorientable.

Our proof will make use of Morse theory of the Yang-Mills functional. This theory was developed for orientable $\Sigma$ by Atiyah-Bott \cite{ab2} and for nonorientable $\Sigma$ by Ho-Liu \cite{ho2008msf}, \cite{hl2} and Ho-Liu-Ramras \cite{ho-liu-ramras}. The Yang-Mills functional is a smooth map $L: \A \rightarrow \R$ which is nonnegative valued and for which $L^{-1}(0) = \A^{\flat}$. The bundle $E$ admits flat connections precisely when the Chern class $c_1(E) \in H^2(\Sigma)$ is torsion, in which case $\A^{\flat}$ forms the minimizing set for $L$. The functional $L$ determines a Morse stratification 
\begin{equation}\label{utnos}
 \A = \coprod_{\mu \in I} \A_{\mu}.
\end{equation}
by finite codimension submanifolds $\A_{\mu}$, where the index set $I =I(E)$ will be described later and for present purposes may be identified with the nonnegative integers $\mu = 0,1,2,...$. The lowest stratum is called the \emph{semistable stratum} and is denoted $\A_{ss}$. When $c_1(E)$ is torsion, the flat connections include into the semistable stratum, inducing an homotopy equivalence of homotopy quotients 
\begin{equation}\label{steak}
\A^{\flat}_{h\gau} \sim (\A_{ss})_{h\gau}
\end{equation}
where for a $G$-space $X$ we use notation $X_{hG} \cong EG \times_G X$ for the homotopy quotient. By definition, $H_G^*(X) = H^*( X_{hG})$, so (\ref{steak}) in particular implies that $H_{\gau}^*(\A^{\flat}) \cong H_{\gau}^*(\A_{ss})$. These isomorphisms hold integrally, but for other reasons we will always work with rational coefficients.

Consider the filtration of $\A$ by subspaces,

$$ \A_{\leq \mu} = \coprod_{\nu \leq \mu} \A_{\nu}.$$
and the long exact sequence in equivariant cohomology of the pair $(\A_{\leq \mu}, \A_{< \mu})$, where $\A_{< \mu} = \A_{\leq \mu} - \A_{\mu}$. A tubular neighbourhood around $\A_{\mu}$ in $\A_{\leq \mu}$ is homeomorphic to the normal bundle $N_{\mu}$ of $\A_{\mu}$ and the homotopy quotient bundle $(N_{\mu})_{h\gau} \rightarrow (\A_{\mu})_{h\gau}$ is known to be orientable in the examples we will consider ( see \cite{ho-liu-ramras}). The rank of $N_{\mu}$ is called the index of $\A_\mu$ and is denoted $\lambda_{\mu}$. Using excision and the Thom isomorphism we have $$H_{\gau}^{*}( \A_{\leq \mu}, \A_{< \mu}) \cong H_{\gau}^*(N_{\mu}, N_{\mu}^o) \cong H_{\gau}^{* - \lambda_{\mu}}(N_{\mu})$$ where $N_{\mu}^o \subset N_{\mu}$ is the complement of the zero section. We obtain a morphism of long exact sequences:

\begin{equation}\begin{CD}\label{first}	
\xymatrix{\ar[r] & H_{\gau}^{*}( \A_{\leq \mu}, \A_{< \mu })\ar[d]^{\cong} \ar[r]^{\alpha_{\mu}} & H_{\gau}^*( \A_{\leq \mu}) \ar[d] \ar[r] & H_{\gau}^*(\A_{< \mu})  \ar[r]  &\\
 & H_{\gau}^{*-\lambda_{\mu}}(N_{\mu}) \ar[r]^{\beta_{\mu}}  &   H_{\gau}^*(N_{\mu}) &  &}
\end{CD}\end{equation}
where the map $\beta_{\mu}$ is the operation of cup product by the equivariant Euler class $Eul_{\gau}(N_{\mu}) = Eul( (N_{\mu})_{h\gau})$. 

When $\Sigma$ is orientable, Atiyah-Bott \cite{ab2} showed that the Euler classes $Eul_{\gau}(N_{\mu})$ were not zero divisors. Consequently, $\beta_{\mu}$ and hence also $\alpha_{\mu}$ are injective and the long exact sequence of the pair $(\A_{\leq \mu}, \A_{< \mu})$ splits into short exact sequences for all $\mu$. A Morse stratification satisfying this property is called \emph{perfect} and for a perfect stratification the Poincar\'e series of the full space is a sum of contributions from each stratum:

 $$ P_t^{\gau}(\A_{ss}) = P_t^{\gau}(\A) - \sum_{\mu \in I -0} t^{\lambda_{\mu}}P_t^{\gau}(\A_{ \mu})$$

On the other hand, when $\Sigma$ is nonorientable and the rank $n=2$ or $3$, Ho-Liu \cite{ho2008msf} show that $Eul_{\gau}(N_{\mu})$ vanishes so the maps $\beta_{\mu}$ equal zero. Ho and Liu call a Morse stratification with this property $\emph{locally antiperfect}$. It does not follow necessarily that the $\alpha_{\mu}$ equal zero, but if they do we say that the stratification is \emph{antiperfect} and for an antiperfect stratification:

\begin{equation}\label{shap}
	P_t^{\gau}(\A_{ss}) = P_t^{\gau}(\A) + \sum_{\mu \in I -0} t^{\lambda_{\mu}-1}P_t^{\gau}(\A_{\mu})
\end{equation}

In the case $n=2$, Ho-Liu prove that the stratification \emph{is} in fact antiperfect. They prove this by comparing the left hand side of (\ref{shap}), which was previously calculated in \cite{baird2008msf}, and showing that it matches the right hand side of (\ref{shap}).

Ho-Liu also conjecture that the stratification is antiperfect when $n=3$. In this paper we prove this conjecture.

\begin{thm}\label{thethm}
Let $E \rightarrow \Sigma$ be a rank three, complex Hermitian vector bundle over an compact, nonorientable 2-manifold without boundary. Then the Morse stratification of the Yang-Mills functional is $\gau$-equivariantly antiperfect. Consequently, we obtain Poincar\'e series
$$ P_t^{\gau}( \A^{\flat}(E)) = (1+t)^g[ (1+t^3)^g(1+t^5)^g + (1+t^2+t^4)(t^3 +2t^4+t^5)^g] P_t(B U(3))$$
when $\Sigma$ is the connected sum of $g+1$ copies of $\R P^2$.  
 \end{thm}

We note that Theorem \ref{thethm} has already been proven when $\Sigma$ is the projective plane or the Klein bottle using a completely different method in \cite{baird2009mfb}.

It is natural to ask whether this remains true for bundles of rank greater than three. Regrettably the answer is no.

\begin{prop}\label{propfour}
For $E$ of rank greater than three, the Yang-Mills Morse stratification is not $\gau$-equivariantly perfect.
\end{prop}

The holonomy of a flat connection determines a homomorphism from the fundamental group $\pi_1(\Sigma)$ to the unitary group of a fibre. This map produces a well known isomorphism of homotopy quotients
\begin{equation}\label{tehhe}
	Hom(\pi_1(\Sigma), U(n))_{hU(n)} \cong \coprod_{c_1(E)\text{ is torsion}} (\A^{\flat}(E))_{h\gau(E)}
\end{equation}
where the action on the left is by conjugation, and the union on the right is over isomorphism types of $\C^n$-bundles over $\Sigma$ admitting flat connections. A conjecture from the author's thesis \cite{bairdthesis} predicts that the homotopy quotient on the left of (\ref{tehhe}) is equivariantly formal. In other words, if $X = Hom(\pi_1(\Sigma), U(n))$, then 
$$ H_{U(n)}^*(X) \stackrel{?}{\cong} H^*(X) \otimes H^*(BU(n))$$
as modules over $H^*(BU(n))$. It is easily shown using the Eilenberg-Moore spectral sequence that equivariant formality is equivalent to $H^*_{U(n)}(X)$ being a free module over $H^*(BU(n))$. 
This conjecture is known to hold for all $\Sigma$ when $n=1,2$ \cite{baird2008msf}; for all $n$ when $\Sigma = \R P^2$\cite{baird2009mfb}; and for $n=3$ when $\Sigma$ is the Klein bottle \cite{baird2009mfb}. The Poincar\'e series computed in Theorem \ref{thethm} provides strong evidence for this conjecture in the case $n=3$, but we have so far been unsuccessful in proving this case. In \S \ref{module} we use the Morse stratification to prove the following partial result:

\begin{thm}\label{thm2}
For $\Sigma$ a compact nonorientable surface without boundary, the equivariant cohomology $H^*_{U(3)}( Hom(\pi_1(\Sigma), U(3)))$ is a torsion-free module over $H^*(BU(3))$.
\end{thm}

\ignore{Theorem \ref{thm2} implies that for $T \subset U(3)$ a maximal torus with Weyl group $W$, the fixed point localization map

$$ H_{U(3)}^*(Hom(\pi_1(\Sigma), U(3))) \rightarrow H_{T}( Hom(\pi_1(\Sigma), U(3))^T)^W$$
is injective, so one might try to }

The strategy behind the proofs of both theorems is to exploit the known results in rank $2$. A rank $3$ bundle $E$ can be decomposed into a sum of a rank $1$ and rank $2$ bundle in two topologically distinct ways. This induces a map of stratified spaces that can be used to understand the boundary maps between strata.

\ignore{\textbf{Notation:}
...}

\textbf{Acknowledgements:}
I want to thank Nan Kuo Ho, Frances Kirwan, Melissa Liu and Dan Ramras for helpful discussions, and Dan Ramras for commenting on an earlier draft. This work was supported by an NSERC postdoctoral fellowship.

\section{Equivariant spaces}
We will make repeated use of the equivariant cohomology functor on the category of equivariant spaces, and so we review this here.

The category of \emph{equivariant spaces} consists of objects $(K, X, \rho)$ where $K$ is a topological group, $X$ is a topological space and $\rho: K \times X \rightarrow X$ is a continuous action, and morphisms $(h,\phi): (K,X,\rho) \rightarrow (K',X', \rho')$ where $h: K \rightarrow K'$ is a continuous group homomorphism, $\phi:X\rightarrow X'$ is a continuous map and $ \rho'\circ (h\times \phi) = \phi \circ \rho$. We will usually omit the action $\rho$ from notation when there is little risk of confusion.

Taking homotopy quotients  defines a covariant functor $(K,X) \mapsto X_{hK}$ from equivariant spaces to the homotopy category of topological spaces, and then applying cohomology determines a contravariant functor $(K,X) \mapsto H_K^*(X)$ to the category of graded commutative rings.

When $(K,X)$ forms a principal $K$-bundle, the homotopy quotient $X_{hK}$ is homotopy equivalent to the orbit space $X/K$. More generally, if $H \unlhd K$ is a normal subgroup, for which $(H,X)$ forms a principal bundle, then $X/H$ inherits a $K/H$-action and $X_{hK}$ is homotopy equivalent to $(X/H)_{h(K/H)}$.

For any equivariant space $(K,X)$, there is a unique morphism to the trivial $K$-space $(K,pt)$. The induced homomorphism $\phi: H_K^*(pt) \rightarrow H_K^*(X)$ naturally makes $H^*_K(X)$ into a module over $H_K^*(pt) = H^*(BK)$. A morphism $(K,X) \rightarrow (K', X')$ induces a commutative diagram of ring homomorphisms

\begin{equation}\begin{CD}
	\xymatrix{ H_K^*(X) & H_{K'}^*(X') \ar[l]_{\phi^*}  \\
	           H^*(BK) \ar[u] & H^*(BK')\ar[l] \ar[u]}
\end{CD}\end{equation}
which allows us to consider $\phi^*$ as a morphism of $H^*(BK')$ modules. 

If $K$ is a connected compact group and $H \subset K$ is a subgroup containing a maximal torus, then the induced map $H^*(BK) \rightarrow H^*(BH)$ is an injective, free extension of $H^*(BK)$.
\ignore{
When $K = U(1)^n$ is a torus, the cohomology ring $ H^*(BU(1)^n)$ is a cohomology ring $\C[x_1,...,x_n]$, where the generators $x_i$ have degree 2. The generators $x_i$ determine a basis for the dual of the Lie algebra $(\lie{u}(1)^n)^* \cong H^2(BU(1)^n)$. A homomorphism $\phi: U(1)^n \rightarrow U(1)^m$, induces a ring homomorphism $\phi^*: H^*(BU(1)^m) \rightarrow H^*(BU(1)^n)$ which is determined by the tangent map at the identity $(d_{\id}\phi)^*: \lie{u}(1)^m \cong H^2(BU(1)^m) \rightarrow \lie{u}(1)^n \cong H^2(BU(1)^n)$.

When $K = U(n)$, the inclusion $U(1)^n \hookrightarrow U(n)$ as a maximal torus induces an injective map $H^*(BU(n)) \rightarrow H^*(BU(1)^n)$ and identifies $H^*(BU(n))$ with the symmetric polynomials $\C[x_1,...,x_n]^{S_n}$. Under this inclusion, we have an isomorphism of free $H^*(BU(n))$-modules
\begin{equation}
	\C[x_1,...,x_n] \cong H^*( Fl(\C^n))\otimes H^*(BU(n))
\end{equation}
where $H^*(Fl(\C^n))$ is the cohomology of the space of complete flags in $\C^n$.
}

\section{Morse stratification for orientable $\Sigma$}

We summarize the relevant material from \S 7 of Atiyah-Bott \cite{ab2}.

Let $\Sigma$ be a oriented 2-manifold without boundary with $E \rightarrow \Sigma$ a Hermitian $\C^n$-bundle. A choice of complex structure on $\Sigma$ produces a homeomophism

\begin{equation}\label{hol}
 \A(E) \cong \CC(E)
\end{equation}
where $\CC(E)$ is the space of holomorphic structures on $E$. The isomorphism (\ref{hol}) sends a connection $\theta$ with covariant derivative $d_{\theta}$ to the holomorphic structure with $\dbar$-operator $\dbar_{\theta} = d_{\theta}^{0,1}$. The Yang-Mills Morse stratification $$ \A = \coprod_{\mu \in I} \A_{\mu}$$admits a holomorphic description, with strata indexed by the set $I=I(E)$ of Harder-Narasimhan types of holomorphic structures on $E$, as we now explain.

The slope of a complex vector bundle $ E \rightarrow \Sigma$, denoted $\slope(E)$, is defined to be the degree of $E$ divided by the rank of $E$, i.e. $$\slope(E) = \deg(E)/\rk(E). $$ A holomorphic structure $\dbar$ on $E$ is called \emph{semistable} if all holomorphic subbundles $F \subset E$ satisfy $ \slope(F) \leq \slope(E)$. Harder and Narasimhan \cite{hn} show that any holomorphic bundle possesses a canonical filtration by holomorphic subbundles $E_1 \subset E_2 \subset ... \subset E_k$ such that the quotients $D_i = E_i/E_{i-1}$ are semistable and 
$$ \slope(D_1) > \slope(D_2) > ... > \slope(D_k).$$ 
The Harder-Narasimhan type or HN-type of $E$ is an $n$-tuple of rational numbers $\mu(E) \in \Q^n$ encoding the ranks and slopes of the subquotients $D_i$. That is,  $\mu = (\mu_1,...,\mu_n)$, with the first $\rk(D_1)$ entries equal to $\slope(D_1)$, the next $\rk(D_2)$ entries equal to $\slope(D_2)$, and so on. The Yang-Mills Morse stratification of $\A(E)$ is indexed by the set of possible HN-types for $E$ and $$ \A(E)_{\mu} = \{ \theta \in \A(E)| \mu(E,\dbar_{\theta})= \mu\}$$
The semistable stratum will also be labelled $\A(E)_{ss}$.

For a given type $\mu$, a choice of $C^{\infty}$-splitting of $E = D_1 \oplus ... \oplus D_k$ determines an injective map 

\begin{equation}\label{loch}
	\prod_{i=1}^k\A(D_i)_{ss} \hookrightarrow \A(E)_{\mu}
\end{equation}
which induces a homotopy equivalence of homotopy quotients 

\begin{equation}\label{loch2}
\prod_{i=1}^k(\A(D_i)_{ss})_{h\gau(E_i)} \sim (\A(E)_{\mu})_{h\gau(E)}.
\end{equation}
This enables us to describe the higher strata of $\A(E)$ in terms of the semistable strata of lower rank bundles.

Each stratum $\A_{\mu} \subset \A$ is a finite codimension submanifold with orientable normal bundle $N_{\mu} \rightarrow \A_{\mu}$. If a connection $\theta$ lies in the image of (\ref{loch}) then $(E, \dbar_{\theta})$ decomposes into a sum of holomorphic bundles $ D_1 \oplus ... \oplus D_n$ and 
\begin{equation}\label{lfd}
N(E)_{\mu, \theta} \cong \bigoplus_{i<j} H^1(\Sigma, Hom(D_i,D_j)).
\end{equation}
This equation is derived by recalling the infinitesimal deformations of a holomorphic structure on $E$ are classified by $H^1(\Sigma, End(E))$ and then deducing that the right hand of (\ref{lfd}) classifies those deformations that disrupt the Harder-Narasimhan filtration. The complex rank can be computed using Riemann-Roch and is given by the formula

\begin{equation}
\rk_{\C} N(E)_{\mu} = \sum_{\mu_i>\mu_j} (\mu_i -\mu_j + (g-1)).
\end{equation}

\section{Morse stratification for nonorientable surfaces}
We summarize the relevant material from Ho-Liu \cite{ho2008msf}, \cite{hl2}, and Ho-Liu-Ramras \cite{ho-liu-ramras}. 

Let $\Sigma$ be a nonorientable compact surface without boundary (i.e. a connected sum of copies of $\R P^2$) and let $\pi: \ti{\Sigma} \rightarrow \Sigma$ denote the orientation double cover with deck transformation $\tau \in Aut(\ti{\Sigma})$. A choice of conformal structure on $\Sigma$ lifts to determine a complex structure on $\ti{\Sigma}$ for which $\tau$ is antiholomorphic.
Let $E \rightarrow \Sigma$ be a Hermitian $\C^n$-bundle. The pullback bundle $ \pcov = \pi^*E$ has vanishing Chern class and thus is trivial. The deck transformation $\tau$ induces an antiholomorphic involution of $\A(\ti{E})$, and pulling back connections induces an isomorphism 
\begin{equation}\label{lang}
\pi^*: \A(E) \cong \A(\ti{E})^{\tau} \subset \A(\pcov).
\end{equation}
between $\A(E)$ and the fixed point set of the involution. 
We identify $\A(E)$ with its image $\A(\ti{E})^{\tau}$, which is a real subspace of $\A(\ti{E})$. Because $\tau$ leaves both the metric and Yang-Mills functional invariant on $\A(\ti{E})$, the Yang-Mills functional produces a Morse stratification of $\A(E)$ with index set $I(E) \subset I(\pcov)$ and strata $$\A(E)_{\mu} = \A(\pcov)_{\mu} \cap \A(E).$$ 
A priori $\A(E)_{\mu}$ might be disconnected, but this turns out not to be the case. Reversing orientation changes degrees by a minus sign, so $\A(E)_{\mu}$ can only be nonempty if $\mu_i = -\mu_{n-i}$ for all $i=1,...,n$.

Choose an HN-splitting $\ti{D}_1\oplus ... \oplus \ti{D_k} \cong \ti{E}$. Then as a $C^{\infty}$-bundle $$E \cong E_0 \oplus \bigoplus_{\slope(\ti{D}_i)> 0} \pi_*\ti{D}_i $$
where $E_0$ is a bundle over $\Sigma$ with rank equal to the summand $\ti{D}_i$ of zero slope and $\pi_*\ti{D}_i$ denotes the pushforward of $\ti{D}_i$ by $\pi$.
This determines a homotopy equivalence of homotopy quotients

\begin{equation}\label{jeff}
  (\A(E_0)_{ss})_{h\gau(E_0)} \times \prod_{\slope(\ti{D}_i)>0} (\A(\ti{D}_i)_{ss})_{h\gau(\ti{D}_i)}  \stackrel{\sim}{\rightarrow} (\A(E)_{\mu})_{h\gau(E)}.
\end{equation}
Equation (\ref{jeff}) follows easily from Prop. 7.3 in \cite{hl2}, and can also be proved in similar fashion to Atiyah and Bott's proof of (\ref{loch2}). Notice that (\ref{jeff}) allows us to describe higher strata $\A(E)_{\mu}$ in terms of lower rank semistable strata over both $\Sigma$ and $\ti{\Sigma}$.

The subspace $\A(E)_{\mu} \subset \A(E)$ is a submanifold with orientable normal bundle $N(E)_{\mu} \rightarrow \A(E)$. The bundle $N(E)_{\mu}$ is equal to the fixed points of the induced antiholomorphic involution $\tau$ on the pullback of $N(\pcov)_{\mu}$ 

$$ N(E)_{\mu} = (N(\pcov)_{\mu}|_{\A(E)})^{\tau}. $$
consequently, the real rank of $N(E)_{\mu}$ equals the complex rank of $N(\pcov)_{\mu}$
\begin{equation}\label{index}
\lambda_{\mu}:= \rk_{\R}(N(E)_{\mu}) = \sum_{\mu_i>\mu_j} (\mu_i -\mu_j + (g-1)).
\end{equation}
For a connection $\theta$ lying in the image of (\ref{jeff}), the fibre $N(\pcov)_{\mu,\theta}$ decomposes as in (\ref{lfd}) and the involution $\tau$ sends summand $H^1(\cov; Hom(\ti{D}_i,\ti{D}_j))$ to $H^1(\cov; Hom(\ti{D}_{k-j}, \ti{D}_{k-i}))$. The summands of $N_{\mu}$ thus fall into two types. When $i \neq k-j$, $\tau$ interchanges two summands of $N(\ti{E})_{\mu,\theta}$, and

\begin{equation}\label{tip1}
[H^1(\Sigma, Hom(\ti{D}_i,\ti{D}_j)) \oplus H^1(\Sigma, Hom(\ti{D}_{k-j},\ti{D}_{k-i}))]^{\tau} \cong H^1(\Sigma, Hom(\ti{D}_{i}, \ti{D}_{j}))
\end{equation}
which has invertible Euler class, by the usual Atiyah-Bott argument. When $i = k-j$, $\tau$ preserves the summand of $N(\ti{E})_{\mu}$, producing 
\begin{equation}\label{tip2}
(H^1(\Sigma, Hom(\ti{D}_i,\ti{D}_{j})))^{\tau}
\end{equation}
which by \cite{ho2008msf} has vanishing Euler class whenever $\ti{D}_i$ has rank one.

\subsection{Rank 2}\label{shead}

Let $\Sigma$ equal a connected sum of $g+1$ copies of $\R P^2$ and $E\rightarrow \Sigma$ a rank $2$ Hermitian bundle with Chern class $c_1(E) = d \text{  (mod } 2)$ in $H^2(\Sigma) \cong \Z/2\Z$. The index set is

$$I(E) = \{(0,0)\} \cup \{ (r,-r)| r > 0 \text{ and }r = d +g+1 \text{  (mod } 2) \}.$$
and the stratum $\A(E)_{(r,-r)}$ satisfies the homotopy equivalence $$ (\A(E)_{(r,-r)})_{h\gau(E)} \sim(\A(\ti{D}_r)_{ss})_{h\gau(\ti{D}_r)}$$ when it is nonempty, where $\ti{D}_r$ is a line bundle over $\ti{\Sigma}$ of degree $r$. Because $\ti{D}_r$ is rank one, all connections are semistable so $ \A(\ti{D}_r) = \A(\ti{D}_r)_{ss} \sim \A^{\flat}(\ti{D}_r)$ and we obtain
$$ (\A(E)_{(r,-r)})_{h\gau(E)} \sim (S^1)^{2g}\times BU(1).$$
For a choice of complex structure on $\ti{\Sigma}$, the factor $(S^1)^{2g}$ may be identified with $Pic_r(\Sigma)$, the moduli space of holomorphic line bundles of degree $r$ on $\Sigma$.

For connection $\theta \in \A(E)_{(r,-r)}$ lying in the image of (\ref{jeff}), the pullback to $\ti{\Sigma}$ decomposes holomorphically as $$(\pi^*E; \dbar_{\pi^*\theta}) = \ti{D}_r \oplus \ti{D}_{-r}$$ where $\ti{D}_r, \ti{D}_{-r}$ are holomorphic line bundles of degree $r,-r$ respectively. The normal bundle at $\theta$ satisfies $N(\pcov)_{(r,-r),\theta} = H^1(\cov ; Hom(\ti{D}_r, \ti{D}_{-r}))$, so

\begin{equation}\label{hat2}
N(E)_{(r,-r),\theta} = H^1(\cov ; Hom(\ti{D}_r, \ti{D}_{-r}))^{\tau},	
\end{equation} 
and $N(E)_{(r,-r)}$ has trivial $\gau$-equivariant Euler class and real rank $2r+g-1$.

\subsection{Rank 3}\label{rnk3case}

For a Hermitian $\C^3$-bundle $E\rightarrow \Sigma$, the index set is

\begin{equation}
I(E) = \{ (r,0,-r)| r=0,1,2,...\}
\end{equation}
and the stratum $\A(E)_{(r,0,-r)}$ satisfies a homotopy equivalence $$ (\A(E)_{(r,0,-r)})_{h\gau(E)} \sim \A(\ti{D}_r)_{\gau(\ti{D}_r)} \times \A(E_0)_{\gau(E_0)}$$
where $\ti{D}_r$ is a line bundle of degree $r$ over $\ti{\Sigma}$ and $E_0$ is a line bundle over $\Sigma$. Because both $\ti{D}_r$ and $E_0$ are line bundles, we get a simpler description by restricting to flat connections:

$$ (\A(E)_{(r,0,-r)})_{h\gau(E)} \sim (S^1)^{2g} \times BU(1)\times(S^1)^{g} \times BU(1) \cong (S^1)^{3g}\times BU(1)^2. $$
For connection $\theta \in \A(E)_{(r,0,-r)}$ lying in the image of (\ref{jeff}), the holomorphic pullback bundle decomposes as $$(\pi^*E; \dbar_{\pi^*\theta}) = \ti{D}_r \oplus \ti{D}_0 \oplus \ti{D}_{-r}$$ where $\ti{D}_r,\ti{D}_0, \ti{D}_{-r}$ are holomorphic line bundles of degree $r,0,-r$ respectively. The normal bundle at $\theta$ decomposes as
\begin{equation}\label{hat3}
N(E)_{(r,0,-r),\theta} \cong H^1(\cov ; Hom(\ti{D}_r, \ti{D}_{-r}))^{\tau}\oplus H^1(\cov;Hom(\ti{D}_r,\ti{D}_0)).	
\end{equation}
The first summand has vanishing equivariant Euler class and real rank $2r+g-1$, while the second summand has real rank $2r+2(g-1)$ and an equivariant Euler class which is not a zero divisor.

\section{Spectral sequence of an ordered stratification}

Let $M$ denote a manifold, possibly infinite dimensional. Suppose that $M$ decomposes as a disjoint union $ M = \coprod_{r=0}^{\infty} M_r,$ where $M_r \subset M$ is a submanifold with finite codimension $\lambda_r$ and with orientable normal bundle $N_r \rightarrow M_r$. Suppose further that the closure $\bar{M}_r \subset \bigcup_{s\geq r} M_p$ and that $\lambda_r \geq r +c$ for some constant $c$. 

The subspaces $M_{\leq p} = \cup_{r=0}^p M_r$ form a topological filtration of $M$ and determines a spectral sequence with $E_1$ page, 
\begin{equation}\label{E1}
E_1^{p,q} = H^{p+q}(M_{\leq p},M_{< p}) \cong H^{p+q-\lambda_p}(M_{p})
\end{equation}
converging to $H^{p+q}(M)$. The second isomorphism in (\ref{E1}) is the Thom isomorphism for the normal bundle of $M_p$. The $r$th (co)boundary operator sends $$d_r^{p,q}: E_r^{p,q} \rightarrow E_r^{p+r, q-r+1}.$$
and the spectral sequences satisfies $E_{r+1}^{p,q} = \ker(d_r^{p,q})/\im(d_r^{p-r, q+r-1})$.

It is easy to see that a stratification is perfect if and only if the spectral sequence collapses at $E_1$, i.e. all boundary maps $d_r^{p,q} =0$. Antiperfect stratifications can also be characterized using properties of the spectral sequence.

\begin{prop}\label{ssver}
Suppose for each $r$, that $H^*(M_r)$ is finite dimensional in each degree. 
The stratification $M = \coprod_{r=0}^{\infty}M_r$ is antiperfect if and only if\\
\\
(1) $i^*: H^*(M) \rightarrow H^*(M_0)$ is injective.\\
(2) $d_r^{p,q}$ is zero when $p,r>0$.\\
\\
where $i: M_0 \hookrightarrow M$ is inclusion.
\end{prop}

\begin{proof}
It follows easily from the definition that the stratification is antiperfect if and only if $P_t(M_0) = P_t(M) + \sum_{p>0}t^{\lambda_p -1}P_t(M_p)$. This is equivalent to $\dim( E_1^{0,*}) = \dim(E_{\infty}^{0,*}) + \sum_{p>0} \dim(E^{p, q-*}_1)$ for all *, which in turn is equivalent to:\\
(1') $d_r^{0,q}$ is surjective for all $r,q$.\\
(2) $d_r^{p,q}$ is zero when $p>0$.\\
so it remains to show that (1) and (2) is equivalent to (1') and (2).

Condition (1') implies that $E_{\infty}^{p,q} = 0$ for $p>0$ which is equivalent to (1). In the other direction, (1) implies that $E_{\infty}^{p,q} = 0$ for $p>0$, while (2) implies that $ E_{\infty}^{p,q} = \mathrm{coker} ( d_p^{0,p+q-1 })$, so (1) and (2) imply (1').
\end{proof}

\begin{cor}\label{ghro}
If $M = \coprod_{r=0}^{\infty}M_r$ is antiperfect, then the filtration of the pair $(M,M_0)$ by subpairs $( M_{\leq p}, M_0)$ induces a spectral sequence that collapses on page 1. Consequently, the filtration of $H^*(M,M_0)$ by kernels, $\ker( H^*(M,M_0) \rightarrow H^*(M_{\leq p}, M_0) )$ has associated grading satifying $$\gr^p(H^*(M,M_0)) \cong H^{*-{\lambda_{\mu}}}(M_p).$$
\end{cor}

\begin{proof}
By construction the map of pairs $(M, \emptyset) \rightarrow (M, M_0)$, respects filtrations and induce a morphism of spectral sequences $\phi: E_*^{*,*}(M,M_0) \rightarrow E_*^{*,*}(M)$. On page 1, and column $p>0$, $\phi$ is an isomorphism $E_2^{p,*}(M,M_0) \cong E_2^{p,*}(M) \cong H^{*-\lambda_{\mu}}(M_p)$, while $E_2^{0,p}(M,M_0) = 0$. Property (2) from Proposition \ref{ssver} implies that the spectral sequence for $(M,M_0)$ collapses on page 1.
\end{proof}

\begin{rmk}\label{rm1}
It was proven in \cite{baird2008msf}, that for nonorientable $\Sigma$ and for all $n>0$ (indeed all compact connected Lie groups) the map $$H_{\gau}(\A) \rightarrow H_{\gau}(\A_{ss})$$ is an injection, establishing $(1)$ for the Yang-Mills stratification. In the next section, we will prove Theorem \ref{thethm} by establishing condition $(2)$.
\end{rmk}

\section{Antiperfection}\label{five}

Let $\Sigma$ be a connected sum of $g+1$ copies of $\R P^2$. The isomorphism classes of $\C^n$-bundles over $\Sigma$ are classified by their first Chern class $c_1(E) \in H^2(\Sigma) \cong \Z/2 \Z$ and we denote by $E_n^{i}$ a principal bundle with $c_1(E_n^{i}) = i \in \Z/2\Z$. Fix a rank 3 bundle $E = E^d_3$. Then $E$ decomposes into a Whitney sum $E_1^i \oplus E_2^{d-i}$ for $i =0$ or $1$.

Define $$\B^i := \A(E_1^i)_{h\gau(E_1^i)} \times \A(E_2^{d-i})_{h\gau(E_2^{d-i})}$$and $\B = \coprod_{i=0,1} \B^i$. Then $\B$ has a stratification with strata $$\B_{r} := \bigcup_{i=0,1} \A(E_1^i)_{h\gau(E_1^i)} \times (\A(E_2^{d-i})_{( r, -r)})_{h\gau(E_2^{d-i})}.$$ According to \S \ref{shead}, the stratum $\B_r$ nonempty for all $r \geq 0$, and connected for $r>0$.

\begin{lem}\label{wellknown}
The stratification $\B = \coprod_{r}\B_r$ is antiperfect.
\end{lem}
\begin{proof}
It suffices to show that the restricted stratification on each component $\B^i$ is antiperfect. But $\B^i$ is the product $\A(E_1^i)_{h\gau(E_1^i)} \times \A(E_2^{d-i})_{\gau(E_2^{d-i})}$ with stratification pulled back from the Yang-Mills stratification of $\A(E_2^{d-i})_{\gau(E_2^{d-i})}$ which was shown to be antiperfect in \cite{ho2008msf}, so the result follows.
\end{proof}

Let $\A = \A(E)$ and $\gau = \gau(E)$. For each $i$, choose an isomorphism $E_1^i\oplus E_2^{d-i} \cong E$. This induces a map

\begin{equation}
\phi: \B \hookrightarrow \A_{h\gau}
\end{equation}
sending $ \B_r$ to $(\A_{(r,0,-r)})_{h\gau}$.

\begin{lem}\label{tift}
The map $\phi$ restricts to homotopy equivalences $$\B_r \sim (\A_{(r,0,-r)})_{\h \gau}$$for all $r>0$.
\end{lem}

\begin{proof}
By (\ref{jeff}) we have a sequence of maps $$ \A(E_1^i)_{h\gau(E_1^i)} \times\A(\ti{D}_r)_{h\gau(\ti{D}_r)} \rightarrow \B_{r} \stackrel{\phi}{\rightarrow} (\A_{(r,0,-r)})_{h\gau} $$ for which the first map and the composed map are both homotopy equivalences. We infer that $\phi$ is also a homotopy equivalence.
\end{proof}

Define filtrations $\B_{\leq p} = \bigcup_{r=0}^p \B_{r}$, and $(\A_{h\gau})_{\leq p} = \bigcup_{r=0}^p (\A_{(p,0,-p)})_{h\gau}$. Then $\phi$ respects these filtrations and so induces a map between the corresponding spectral sequences. 

Let $E_r^{p,q}(B)$ denote the spectral sequences determined by the filtration of $\B$ and let $E_r^{p,q}(A)$ denote the spectral sequences determined by the filtration of $\A_{h\gau}$.

\begin{lem}\label{cuurry}
The filtered map $\phi$ induces a map of spectral sequences $\phi_r^{p,q}: E_r^{p,q}(A) \rightarrow E_r^{p,q}(B)$ for which $\phi_1^{p,q}$ is injective when $p>0$.
\end{lem}

\begin{proof}	

The map $\phi_r^{p,q}$ fits into a commutative diagram
\begin{equation}\begin{CD}
\xymatrix{E_1^{p,q}(A)  \ar[r]^{\phi_r^{p,q}} & E_1^{p,q}(B)\\
H^{p+q}((\A_{h\gau})_{\leq p}, (\A_{h\gau})_{< p}  ) \ar[r] \ar[u]^{=}& H^{p+q}( \B_{\leq p} , \B_{<p}) \ar[u]^{=}\\
H^{p+q - \lambda_p(a)}((\A_{(p,0,-p)})_{h\gau}) \ar[d]_{\cong} \ar[r] \ar[u]^{\cong} & H^{p+q - \lambda_p(b)}( \B_{p}) \ar[u]^{\cong} \ar[d]_{\cong} \\
H^{p+q - \lambda_p(a)}( \mZ_{p})  \ar[r]^{\alpha}  & H^{p+q - \lambda_p(b)}( \mZ_{p})  
}
\end{CD}\end{equation}
where $\mZ_p$ is $\A(E_1^i)_{h\gau(E_1^i)} \times\A(\ti{D}_p)_{h\gau(\ti{D}_p)}$ an includes into $(\A_{(p,0,-p)})_{h\gau}$ and $\B_p$ as described in the proof of Lemma \ref{tift}, and $\lambda_p(a)$ and $\lambda_p(b)$ are the codimension of the $p$th stratum in $\A_{h\gau}$ and $\B$ respectively. So we are reduced to proving that $\alpha$ is an injection. But $\alpha$ is simply cup product by the equivariant Euler class of the quotient bundle $ N_p(A)|_{\mZ_{p}} /N_p(B)|_{\mZ_p}$. Comparing (\ref{hat2}) and (\ref{hat3}), we see that $ N_p(A)|_{\mZ_{p}} /N_p(B)|_{\mZ_p}$ is isomorphic to the second summand of (\ref{hat3}) which is not a zero divisor, so $\alpha$ is injective.
\end{proof}

\begin{proof}[Proof of Theorem \ref{thethm}]
Denote the boundary operators of the spectral sequences $E_r^{p,q}(A)$ and $E_r^{p,q}(B)$ by $d_r^{p,q}(A)$ and $d_r^{p,q}(B)$ respectively. By Proposition \ref{ssver} and Remark \ref{rm1}, to prove antiperfection, we only need to prove that $d_r^{p,q}(A)=0$ when $p>0$. By Proposition \ref{ssver} and Lemma \ref{wellknown}, we know $d_r^{p,q}(B) = 0$ for $p>0$. By the commutative diagram 

\begin{equation}\begin{CD}\label{injcom}
\xymatrix{ E_r^{p,q}(A) \ar[rr]^{d_r^{p,q}(A)} \ar[d]^{\phi^{p,q}_r} && E_r^{p+r, q-r+1}(A) \ar[d]^{\phi^{p+r,q-r+1}_r}\\
E_r^{p,q}(B) \ar[rr]^{0}  && E_r^{p+r, q-r+1}(B)\\
}	
\end{CD}\end{equation}
it is sufficient to prove that $\phi_r^{p,q}$ is injective for $p>0$. This was proven for $r=1$ in Lemma \ref{cuurry}. Now assume inductively that $\phi_r^{p,q}$ is injective for all $p>0$, all $q$ and a fixed $r$. Then on the $r+1$ page we have

\begin{equation}\label{fixeh}
	\phi_{r+1}^{p,q}: E_{r+1}^{p,q}(A) = \frac{\ker( d_r^{p,q}(A))}{\im( d_r^{p-r,q+r-1}(A))} \rightarrow E_{r+1}^{p,q}(B) = \frac{\ker( d_r^{p,q}(B))}{\im( d_r^{p-r,q+r-1}(B))}
\end{equation}
When $p>0$ we know that $d_r^{p,q}(A) = d_r^{p,q}(B)=0$, so the numerators equal $E_r^{p,q}$. If $p \neq r$ then the denominators are zero and we have $\phi_{r+1}^{p,q} \cong \phi_{r}^{p,q} $ is injective by induction. If $p = r$ then both sides of (\ref{fixeh}) are zero by condition $(1)$ of Proposition \ref{ssver}, which has already been established for both $A$ and $B$.

Applying formula (\ref{shap}), we obtain $P_t^{\gau}(\A^{\flat}) = P_t^{\gau}(\A_{ss}) = P_t^{\gau}(\A) + \sum_{r=0}^{\infty} t^{\lambda_r-1}P_t^{\gau}(\A_{(r,0,-r)})$. All the quantities on the right of this equality are known (see Remark \ref{commentst} and \S \ref{rnk3case}) so the computation of $P_t^{\gau}(\A^{\flat})$ follows directly. 
\end{proof}

\begin{rmk}\label{commentst}
We used at the end of the preceding proof the result from \cite{baird2008msf} \S 4.5, that for $E \rightarrow \Sigma$ a Hermitian $\C^n$-bundle over a connected sum of $g+1$ copies of $\R P^2$, we have $H^*_{\gau(E)}(\A(E)) \cong H^*_{U(n)}( U(n)^g) \cong H^*(U(n)^g)\otimes H^*(BU(n))$ as $H^*(BU(n))$ modules, where $U(n)$ acts diagonally by conjugation.
\end{rmk}

\section{Module structure}\label{module}
In this section, we prove that the action of $H^*(BU(3))$ on $H^*_{U(3)}(Hom(\pi_1(\Sigma), U(3)))$ is torsion free. Throughout, let $\Sigma$ denote a connected sum of $g+1$ copies of $\R P^2$.

\subsection{Homomorphisms of $\pi_1(\Sigma)$ and connections}

It was stated in the introduction that we have an isomorphism of rings 

\begin{equation}\label{teghe}
H_{U(3)}^*(Hom(\pi_1(\Sigma), U(3))) \cong \bigoplus_{c_1(E) \text{ is torsion} } H_{\gau(E)}^*(\A(E)_{ss}).
\end{equation}
This isomorphism is induced by a morphism of equivariant spaces as we now explain.

Choose a base point $x\in \Sigma$, we define $\gau_{0} = \gau_0(E) \subset \gau(E)$ to be the group of gauge transformations of $E$, which leave the fibre $E_x$ fixed. Then $\gau_{0}$ fits into a short exact sequences $$1 \rightarrow \gau_{0} \rightarrow \gau \rightarrow U(E_x) \rightarrow 1$$ 
where $U(E_x) \cong U(n)$ is the group of unitary transformations of the fibre. Combined with the quotient map $\A \rightarrow \A/\gau_{0}$, this induces a morphism $ \A_{h \gau} \rightarrow (\A/ \gau_0)_{hU(n)}$ by functoriality. Because the based gauge group acts freely on $\A$, this is an isomorphism $ \A_{h \gau} \cong (\A/\gau_{0})_{hU(n)} $ and similarly for all equivariant subspaces of $\A$.

Holonomy defines a map $\A^{\flat}(E) \rightarrow Hom(\pi_1(\Sigma,x), U(E_x))$ which descends to a homeomorphism 
\begin{equation}\label{gropuid}
\coprod_{c_1(E) \text{ is torsion}} \A^{\flat}(E)/\gau_0(E) \cong Hom(\pi, U(n))
\end{equation}
where $\pi = \pi_1(\Sigma, x)$ and we have identified $U(E_x) \cong U(n)$. The isomorphisms (\ref{tehhe}) and (\ref{teghe}) follow.

\subsection{Rank 2}

Before considering the rank 3 case, it will be important to understand what happens in rank $2$. Let $E = E_2 \rightarrow \Sigma$ be a rank 2 Hermitian bundle. 

\begin{lem}\label{scoe}
For $E \rightarrow \Sigma$ of rank 2, the cohomology of the pair $H_{\gau(E)}^*(\A, \A_{ss}) \cong H_{U(2)}^*(\A/\gau_0, \A_{ss}/\gau_0)$ is a free module over $H^*(BU(2))$. 
\end{lem}

\begin{proof}
It was proven in \cite{baird2008msf} that equivariant formality holds for $H^*( Hom(\pi_1(\Sigma), U(2))) \cong H^*_{\gau}(\A_{ss})$, so it is in particular free over $H^*(BU(2))$. Consider now the short exact sequence 
\begin{equation}\label{sss}
0 \rightarrow H_{\gau}^*( \A) \rightarrow H_{\gau}^*(\A_{ss}) \rightarrow H^*_{\gau}(\A, \A_{ss})\rightarrow 0.
\end{equation} 
In the analogous situation where $U(2)$ is replaced by $SU(2)$, (\ref{sss}) was shown to split in \cite{baird2008msf} (see proof of Theorem 1.1), the argument being that the image of $H^*_{\gau}(\A)$ coincides with the fixed points of a $\Z/2\Z$-automorphism of $H^*_{\gau}(\A_{ss})$ (which is defined by tensoring by a flat line bundle). The same argument works here, and we deduce that $H^*_{\gau}(\A, \A_{ss})$ is projective and hence free over $H^*(BU(2))$.
\end{proof}

We deduce as a simple consequence,
\begin{cor}
The ring $H^*(\B, \B_{ss})$ is a free module over $H^*(BU(3))$.
\end{cor}

\begin{proof}
We have $H^*(\B, \B_{ss}) = \oplus_{i=0,1} H^*(\B^i, \B^i_{ss})$, where $(\B^i, \B^i_{ss}) = \A(E_1^i)_{h\gau(E_1^i)} \times (\A(E_2^{d-i})_{h\gau(E_2^{d-i})}, (\A(E_2^{d-i})_{ss})_{h\gau(E_2^{d-i})})$. By Lemma \ref{scoe} and the fact that $H^*_{\gau(E_1^i)}(\A(E_1^i)) \cong H^*(U(1)^g) \otimes H^*(BU(1))$, it is clear that $H^*(\B, \B_{ss})$ is a free module over $H^*(BU(1) \times BU(2))$. Because $U(3)$ and $U(1)\times U(2)$ share a maximal torus, $H^*(BU(3)) \rightarrow H^*(BU(1) \times BU(2))$ is a free extension, and the result follows.
\end{proof}

\subsection{Proof of Theorem \ref{thm2}}

We now turn our attention to the case $E = E_3^d$ is a rank three bundle over nonorientable $\Sigma$. Once again, we have a short exact sequence $$ 0 \rightarrow H^*_{\gau}(\A) \rightarrow H^*_{\gau}(\A_{ss}) \rightarrow H^*_{\gau}(\A, \A_{ss}) \rightarrow 0 $$
and by Remark \ref{commentst} $H^*_{\gau}(\A) \cong H^*(U(3)^g) \otimes H^*(BU(3))$ is a free module over $H^*(BU(3))$. Thus to prove Theorem \ref{thm2} it suffices to prove that $H^*_{\gau}(\A, \A_{ss})$ is torsion free over $H^*(BU(3))$.

To prove this we will use that map of pairs $ (\B, \B_{ss})\hookrightarrow (\A_{h\gau}, (\A_{ss})_{h\gau})$. This map is filtered by the pairs $(\B_{\leq p}, \B_{ss}) \rightarrow ((\A_{\leq p})_{h\gau}, (\A_{ss})_{h\gau})$, inducing an associated graded map.

\begin{lem}
The cohomology map 
\begin{equation}\label{slep}
H_{\gau}^*( \A, \A_{ss}) \rightarrow H^*(\B, \B_{ss})
\end{equation}
 is injective. Consequently, $H_{\gau}^*(\A, \A_{ss})$ is torsion free. 
\end{lem}

\begin{proof}
By Corollary \ref{ghro} and Lemma \ref{cuurry}, the associated graded version of (\ref{slep}) is injective. It follows that (\ref{slep}) is also injective. Since $H_{\gau}^*(\A, \A_{ss})$ injects into a free $H^*(BU(3))$-module, it must be torsion free.
\end{proof}

\section{ Failure of antiperfection in higher rank}

Consider now the case that $E$ has rank $n >3$. Set $ M := H^*_{\gau(E)}( \A(E)_{ss})$. By (\ref{gropuid}), $M \cong H^*_{U(n)}(X)$, where $X$ is one of the two components of $Hom( \pi, U(n))$ depending on the Chern class of $E$. In particular, since $X$ is compact, $M$ is a finitely generated module for the polynomial ring $A := H^*(BU(n))$. By Hilbert's syzygy theorem, there exists a finite length free resolution 
\begin{equation}\label{syzygy}
 0 \rightarrow F_n \rightarrow ... \rightarrow F_0 \rightarrow M,
\end{equation}
where each $F_i$ is a direct sum of a finite number of degree shifted copies of $A$. It follows that the Hilbert series of $M$ must be of the form $ p(t) P_t( BU(n)) = p(t)/(1-t^2)(1-t^4)...(1-t^{2n})$ for some polynomial $p(t)$ with integer coefficients and for which $p(1) = \sum_{i=0}^n (-1)^i \rk_A (F_i)$. 

\begin{lem}\label{hugner}
The value of $p(t)$ at $t=1$ is $2^{(g+1)n-1}$.
\end{lem}

\begin{proof}
Let $Q$ denote the quotient field of $A$. Because localization preserves exactness, we may tensor the sequence (\ref{syzygy}) by $Q$ to obtain an exact sequence of vector spaces over $Q$. We deduce the identity $ \dim_Q M\otimes_A Q = \sum_{i=0}^n (-1)^i \dim_Q F_i \otimes_A Q = p(1)$. On the other hand, the Borel localization theorem identifies $\dim_Q M \otimes_A Q$ with the sum of Betti numbers of the space $X^T$, where $T$ is a maximal torus in $U(n)$ and $X^T$ is the locus of $T$-fixed points. 

The space $X^T$ is equal to half the components of $Hom( \pi, U(n))^T$, which because $T$ is maximal abelian can be identified with  $Hom( \pi, T)$. Since $T$ has rank $n$, we have $ Hom( \pi, T) \cong (Hom( \pi, U(1)))^{\times n}$, and $ Hom(\pi, U(1))$ was shown in \cite{baird2008msf} to have two components, each homeomorphic to $(S^1)^{\times g}$. Thus $X^T$ has $2^{n-1}$ components, each homeomorphic to $(S^1)^{ng}$. Summing the Betti numbers of $X^T$ completes the argument.
\end{proof}

In particular, it follows that

\begin{equation}\label{limit}
\lim_{t \rightarrow 1} P_t^{\gau}( \A_{ss})\prod_{i=1}^n(1-t^{2i}) = 2^{(g+1)n-1}.
\end{equation}
Our proof of Proposition \ref{propfour} proceeds by showing that (\ref{limit}) is inconsistent with (\ref{shap}).

\begin{proof}[Proof of Proposition \ref{propfour}]

Suppose for the sake of contradiction, that the Morse stratification is $\gau$-equivariantly antiperfect. Then (\ref{shap}) implies that,

\begin{equation}\label{check}
	\lim_{t \rightarrow 1} (P_t^{\gau}(\A_{ss}) -P_t^{\gau}(\A))\prod_{i=1}^n(1-t^{2i})=  \lim_{t \rightarrow 1} \sum_{\mu \in I'} t^{\lambda_{\mu}-1}P_t^{\gau}(\A_{\mu})\prod_{i=1}^n(1-t^{2i})
\end{equation}
where $I'$ indexes all nonsemistable strata. Combining Remark \ref{commentst} with (\ref{limit}) shows that (\ref{check}) equals $2^{(g+1)n-1} - 2^{gn}$.

Because $P_t^{\gau}(\A_{\mu})$ is a power series with nonnegative coefficients it is clear that for any subset $J \subset I'$, 
$$ \lim_{t \rightarrow 1} \sum_{\mu \in J} t^{\lambda_{\mu}-1}P_t^{\gau}(\A_{\mu})\prod_{i=1}^n(1-t^{2i})\leq \lim_{t \rightarrow 1} \sum_{\mu \in I'} t^{\lambda_{\mu}-1}P_t^{\gau}(\A_{\mu})\prod_{i=1}^n(1-t^{2i}) $$
whenever the limit on the left exists. Now choose $J = \{ (d,0,0,...,0,-d)| d>0\}$. We will show that the limit of the sum indexed by $J$ has value larger than $2^{(g+1)n-1} - 2^{gn}$ producing a contradiction. 

By (\ref{jeff}), we have $ (\A_{(d,0,...,0,-d)})_{h\gau} \cong \A(\tilde{L}_d)_{h \gau(\ti{L}_d)} \times (\A(E_0)_{ss})_{h\gau(E_0)}$ where $\ti{L}_d$ is a degree $d$ line bundle over the double cover $\ti{\Sigma}$ and $E_0$ is a rank $n-2$ bundle over $\Sigma$. Applying the Kunneth formula,

$$P_t^{\gau}(\A_{(d,0,...,-d)}) = \frac{(1+t)^{2g}}{1-t^2} \frac{q_d(t)}{\prod_{i=1}^{n-2}(1-t^{2i})}$$
where $q_d(t)$ is a polynomial satisfying $q_d(1) = 2^{(g+1)(n-2)-1}$ by Lemma \ref{hugner}. Here $q_d(t)$ depends only on the parity of $d$, so in particular the convergence $q_d(t) \rightarrow q_d(1)$ is uniform in $d$. Applying (\ref{index}), the index of the stratum $\A_{(d,0,...,0,-d)}$ is $\lambda_d := 2d(n-1) + (g-1)(2n-3)$. Calculating:
\[\begin{split} 
	\lim_{t \rightarrow 1} \sum_{\mu \in J} t^{\lambda_{\mu}-1}P_t^{\gau}(\A_{\mu})\prod_{i=1}^n(1-t^{2i}) &= \lim_{t\rightarrow 1} \frac{(1-t^{2n-2})(1-t^{2n})(1+t)^{2g}}{1-t^2} \sum_{d=1}^{\infty} q_d(t)t^{\lambda_d}\\ &= n 2^g q_1(1)= n 2^{(g+1)n -3}
\end{split}\]
which is greater than $2^{(g+1)n-1} - 2^{gn}$ when $n \geq 4$.

\end{proof}

\bibliographystyle{plain}

\bibliography{TomReferences}

\end{document}